\documentclass[12pt]{amsart}
\usepackage{geometry}                
\usepackage{graphicx}
\usepackage{amsmath,amsfonts,amssymb,amsthm}
\usepackage{epstopdf,mathrsfs}
\usepackage{caption,subcaption}
\usepackage[colorlinks,citecolor=blue,urlcolor=blue]{hyperref}
\usepackage{tikz}
\usetikzlibrary{chains,matrix,scopes,fit,decorations,positioning,shapes,snakes,arrows,decorations.markings}
\usepackage[algo2e,norelsize]{algorithm2e}

\theoremstyle{plain}
\newtheorem{thm}{Theorem}[section]
\newtheorem{lem}[thm]{Lemma}
\newtheorem{prop}[thm]{Proposition}
\newtheorem{cor}[thm]{Corollary}

\theoremstyle{definition}
\newtheorem{defn}[thm]{Definition}

\newtheorem{exmp}[thm]{Example}

\theoremstyle{remark}
\newtheorem{rem}[thm]{Remark}

\newcommand{\sC}{\mathscr{C}}

\newcommand{\cE}{\mathcal{E}}
\newcommand{\cG}{\mathcal{G}}
\newcommand{\cH}{\mathcal{H}}
\newcommand{\cK}{\mathcal{K}}

\newcommand{\sP}{\mathscr{P}}
\newcommand{\R}{\mathbb{R}}
\newcommand{\bS}{\mathbf{S}}
\newcommand{\cS}{\mathcal{S}}
\newcommand{\sS}{\mathscr{S}}
\newcommand{\sT}{\mathscr{T}}

\newcommand{\bX}{\mathbf{X}}
\newcommand{\bx}{\mathbf{x}}

\newcommand{\indep}{{\;\bot\!\!\!\!\!\!\bot\;}}

\newcommand{\la}{\langle}
\newcommand{\ra}{\rangle}

\newcommand{\ot}{\leftarrow}
\newcommand{\sgn}{\mathrm{sgn}}
\newcommand{\liea}{\mathfrak{g}}

\newcommand{\RR}{{\mathbb R}}
\newcommand{\NN}{{\mathbb N}}
{\begin{figure} \begin{center}}%
{\end{center} \end{figure}}

\def\sgn{{\rm sgn}\,}

\newcommand{\tcH}{\widetilde{\cH}}

\newcommand{\SplusG}{\cS_\cG^+}

\newcommand{\down}{\hspace{-.5ex}\downarrow\hspace{-.5ex}}
\newcommand{\downb}{\,\down}

\newcommand{\Aut}{\mathrm{Aut}}

\definecolor{darkgreen}{rgb}{0.1,0.6,0.1}

\title[Groups of Gaussian chain graph models]{Automorphism groups of Gaussian chain graph models}

\author{Jan Draisma}
\address{TU Eindhoven\\Department of Mathematics and Computer Science \\PO
Box 513 \\5600 MB Eindhoven \\ and CWI Amsterdam \\ The Netherlands }
\curraddr{}
\email{j.draisma@tue.nl}
\thanks{Both authors acknowledge support of Draisma's Vidi grant
from the Netherlands Organisation for Scientific Research (NWO) and hospitality of the Simons Institute for the Theory of Computing in Berkeley during the Fall 2014 program {\em Algorithms and Complexity in Algebraic Geometry}. PZ was supported from the European
Union Seventh Framework Programme (PIOF-GA-2011-300975). }
\author{Piotr Zwiernik}
\address{Universit\`a di Genova\\Dipartimento di Matematica \\16146 Genova\\Italy}
\curraddr{}
\email{piotr.zwiernik@gmail.com}
\thanks{}

\begin{document}
\maketitle
\begin{abstract}
In this paper we extend earlier work on groups acting on Gaussian graphical models to Gaussian Bayesian networks and more general Gaussian models defined by chain graphs. We discuss the maximal group which leaves a given model invariant and provide basic statistical applications of this result. This includes equivariant estimation, maximal invariants and robustness. The computation of the group requires finding the essential graph.  However, by applying St\'{u}deny's theory of imsets we show that computations for DAGs can be performed efficiently without building the essential graph. In our proof we derive simple necessary and sufficient conditions on vanishing sub-minors of the concentration matrix in the model.\end{abstract}


\section{Introduction}

Having an explicit group action on a parametric statistical model gives a better understanding of equivariant estimation or invariant testing for the model under consideration \cite{nielsen_transformation82,eaton1989,lehmannromano,sun2005estimation}. In \cite{DrKuZw2013} we have identified the largest group that acts on an undirected  Gaussian graphical model and we have shown how this group can be used to study equivariant estimators of the covariance matrix in this model class. In the present paper we extend our discussion  to chain graph models.

A chain graph $\cH$ is a graph with both directed and undirected edges that  contains no semi-directed cycles, that is sequences of nodes $i_1,\ldots,i_k,i_{k+1}=i_1$ such that for every $j=1,\ldots,k$ either $i_j-i_{j+1}$ or $i_j\to i_{j+1}$ but at least one edge is directed. In this paper we focus on \emph{chain graphs without flags} (NF-CGs), that is with no induced subgraphs of the form $i\to j-k$. Note that both undirected graphs and directed acyclic graphs (DAGs) are chain graphs without flags. For more details on these graph-theoretic notions see Section \ref{sec:graph}.

\emph{Gaussian models on chain graphs} constitute a flexible family of graphical models, which contains both undirected Gaussian graphical models and Gaussian Bayesian networks defined by directed acyclic graphs (DAGs). Let $\cH$ be a NF-CG. Let $\R^\cH$ denote the space of all $m\times m$ matrices $\Lambda=[\lambda_{ij}]$ such that $\lambda_{ij}=0$ if $i\not\rightarrow j$ in $\cH$; let $\cS^+_m$ denote the space of all $m\times m$ symmetric positive definite matrices and let $\cS^+_\cH$ be the subspace of  $\cS^+_m$ consisting of matrices $\Omega=[\omega_{ij}]$ such that $\omega_{ij}=0$ if $i\neq j$ and $i-\!\!\!\!\!/\, j$ in $\cH$. The Gaussian chain graph model $M(\cH)$ of a NF-CG $\cH$ consists of all Gaussian distributions on $\R^{m}$ with mean zero and concentration matrices $K$ of the form 
\begin{equation}\label{eq:parameter}
K\quad =\quad (I-\Lambda)\Omega(I-\Lambda^T)\quad\mbox{such that } \Lambda\in \R^\cH,\, \Omega\in \cS^+_\cH.
\end{equation}
The set of all matrices of this form will be denoted by $\cK(\cH)$. 

\begin{rem}
Two non-equivalent definitions of chain graph models can be found in the literature and they are referred to as LWF or AMP chain graph models in \cite{andersson2001}, which refers to: Lauritzen-Wermuth-Frydenberg \cite{frydenberg1990,LauritzenWermuth1989} and Andersson-Madigan-Perlman \cite{andersson2001} (Alternative Markov Properties). These two definitions differ in how exactly a graph encodes the defining set of conditional independence statements.  However, if $\cH$ has no flags then both definitions coincide (see \cite[Theorem 1, Theorem 4]{andersson2001}). 
\end{rem}

Let $\bX$ be a Gaussian vector with the covariance matrix $\Sigma\in M(\cH)$. A linear transformation $g\in {\rm GL}_m(\R)$ yields another Gaussian vector $\mathbf{Y}=g \bX$. A basic question of equivariant inference is for which $g$ the covariance matrix $g\Sigma g^T$ of $\mathbf{Y}$ still lies in $M(\cH)$. More formally, the general linear group ${\rm GL}_{m}(\mathbb{R})$ acts on $\mathcal{S}^{+}_{m}$ by $g\cdot \Sigma:=g\Sigma g^T$. Fix a chain graph $\cH$. We study the problem of finding:
\begin{equation}\label{eq:mainproblem}
G\quad:=\quad\{g\in {\rm GL}_m(\R)|\,g\cdot M(\cH)\subseteq M(\cH)\}.
\end{equation}
In other words, find the stabilizer of $M(\cH)$ in ${\rm
GL}_m(\R)$. 

\begin{rem}
The set $G$ is a closed algebraic subgroup of ${\rm GL}_m(\R)$, and in
particular has the structure of a Lie group. First, it is clear from the
definition that $G$ is closed under matrix multiplication. To see that it
is closed under inversion and closed in the Zariski topology, we argue as
follows.  Let $\overline{M(\cH)}$ denote the Zariski closure of $M(\cH)$
in $\RR^{m \times m}$, that is, the set of matrices in $\RR^{m \times m}$
whose entries satisfy all polynomial equations that hold identically on
$M(\cH)$. Suppose that $g \in {\rm GL}_m(\R)$ maps $\overline{M(\cH)}$
into itself. Then, since acting with $g$ preserves positive definite
matrices and since $M(\cH)$ consists of all positive definite matrices
in $\overline{M(\cH)}$ (see \cite[Proposition 3.3.13]{oberwolfach2009} for the case
of DAGs; the general chain graph case is similar), $g$ also preserves
$M(\cH)$. Thus $G$ may be characterized as the stabilizer of the real
algebraic variety $\overline{M(\cH)}$. This shows that $G$ is Zariski
closed. To see that it is also closed under inversion, note that $g \cdot
\overline{M(\cH)}$ is a real algebraic variety of the same dimension as
$\overline{M(\cH)}$ and contained in $\overline{M(\cH)}$, hence equal
to $\overline{M(\cH)}$. But then also $g^{-1} (\overline{M(\cH)})$
equals $\overline{M(\cH)}$.
\end{rem}

The problem in (\ref{eq:mainproblem}) can be alternatively phrased in terms of concentration matrices, which will be more useful in our case. Let ${\rm GL}_m(\R)$ act on $\cS^+_m$ by $g\cdot K:=g^{-T}K g^{-1}$. Now find all $g\in {\rm GL}_m(\R)$ such that $g\cdot \cK(\cH)\subseteq \cK(\cH)$. 

\begin{exmp}\label{ex:DAG3}
The DAG $\overset{1}{\bullet}\rightarrow\overset{2}{\bullet}\rightarrow\overset{3}{\bullet}$ defines a model given by a single conditional independence statement $X_1\indep X_3|X_2$ and hence is equal to the model on the undirected graph $\overset{1}{\bullet}-\overset{2}{\bullet}-\overset{3}{\bullet}$. Since the directed part of this graph is empty then by (\ref{eq:parameter}) the model  consists of all covariance matrices such that the corresponding concentration matrices are of the form
$$
K=\Omega=\left[\begin{array}{ccc}
* & * & 0\\
* & * & *\\
0 & * & *
\end{array}
\right].
$$
By \cite[Theorem 1.1]{DrKuZw2013} $G$ in this case consists of invertible matrices of the form
$$
\left[\begin{array}{ccc}
* & * & 0\\
0 & * & 0\\
0 & * & *
\end{array}
\right]\qquad\mbox{or}\qquad \left[\begin{array}{ccc}
0 & * & *\\
0 & * & 0\\
* & * & 0
\end{array}
\right].
$$
\end{exmp}

\subsection{The group $G$}

Example \ref{ex:DAG3} showed that two different chain graphs may  define the same chain graph model. We discuss this in more detail in Section \ref{sec:equivalence}. For any NF-CG  $\cH$ denote by $\cH^*$ the \emph{unique} graph without flags with the largest number of undirected edges which induces the same Gaussian model as $\cH$. The fact that such a unique graph exists follows from Proposition \ref{prop:essexists} given later. For example for the DAG in Example \ref{ex:DAG3} such a graph is given by the undirected graph $\overset{1}{\bullet}-\overset{2}{\bullet}-\overset{3}{\bullet}$. By $c^*(i)$ we denote the children of $i$ in $\cH^*$, so $c^*(i)=\{j:\,i\rightarrow j\mbox{ in }\cH^*\}$. Similarly by $n^*(i)$ we denote the set of neighbours of $i$ in $\cH^*$, that is, nodes $j$ connected to $i$ by an undirected edge, which we denote by $i-j$. We write
$$
N^{*}(i)\,\,\,:=\,\,\,\{i\}\cup n^{*}(i)\cup c^{*}(i).
$$

Our main results can be summarized as follows. For a fixed chain graph without flags $\cH$ with set of nodes given by $[m]:=\{1,\ldots,m\}$ consider the set $G^0$ of invertible matrices given by
\begin{equation}\label{eq:G0}
G^{0}\,\,\,:=\,\,\,\{g=[g_{ij}]\in {\rm GL}_{m}(\R):\,\, g_{ij}= 0 \mbox{ if }N^{*}(i)\not\subseteq N^{*}(j)\}.
\end{equation}
Further, an automorphism of a chain graph is any permutation of its nodes that maps directed edges to directed edges and undirected edges to undirected edges. 
 \begin{thm}\label{th:main}Let $\cH$ be a chain graph without flags. The group $G$ in (\ref{eq:mainproblem}) is generated by its connected normal subgroup $G^0$ and the group ${\rm Aut}(\cH^*)$ of automorphisms of the essential graph $\cH^*$. 
\end{thm}
In the undirected case, this theorem reduces to \cite[Theorem 1.1]{DrKuZw2013}. However, the proof in our current, more general setting is much more involved, first because the set $\cK(\cH)$ is not a linear space, and second because the characterization is in terms of the essential graph rather than the graph itself.

Note that for some graphs there may be two nodes $i,j$ such that $N^{*}(i)=N^{*}(j)$. In this case the transposition of $i$ and $j$ lies already in $G^0$, which shows that $G^0$ and ${\rm Aut}(\cH^*)$ may have a non-trivial intersection. In Section \ref{sec:GG} we prove a more refined version of Theorem \ref{th:main} that gets rid of this redundancy.

Given a set of edges defining a chain graph without flags $\cH$ we would like to find $G^{0}$ by listing all pairs $(i,j)$ for $i,j\in [m]$ such that $g_{ij}=0$ for all $g\in G^{0}$. Since our theorem depends on computing the essential graph $\cH^*$, a natural question arises on complexity of this computation. In Section \ref{sec:DAG} we show how $G^{0}$ can be efficiently computed in the case of DAGs. We propose an efficient algorithm  that does not require computing the essential graph $\cH^*$. 

\subsection{Existence and robustness of equivariant estimators}
The description of the group $G$ can be used to analyse the inference for chain graph models. Let $\mathbf{X}=(\mathbf{X}_1,\ldots,\mathbf{X}_n)$ denote a random sample of length $n$ from the model $M(\cH)$. An estimator of the covariance matrix of $\bX$ is any map $T_n:(\R^m)^n\rightarrow M(\cH)$. In this paper we are interested in \textit{equivariant estimators}, that is, estimators satisfying 
\begin{equation}\label{eq:equivariance}
T_n(g\cdot \mathbf{X})=g\cdot T(\mathbf{X})\qquad\mbox{for every }g\in G, \mathbf{X}\in (\R^m)^n, 
\end{equation}
where the action of $G$ on $(\R^m)^n$ is 
$$
g\cdot (x_1,\ldots,x_n)=(gx_1,\ldots,gx_n).
$$
An important example of an equivariant estimator is the maximum likelihood estimator. A natural theoretical question is how large the sample size needs to be so that an equivariant estimator $T_n$ exists with probability one (see \cite[Section 1.2]{DrKuZw2013}). Define $\downarrow\! i:=\{j:\, N^{*}(i)\subseteq N^{*}(j)\}$.

\begin{thm}\label{th:Geqexist}
Let $\cH$ be a chain graph without flags with set of nodes $[m]$. There exists a $G$-equivariant estimator $T_{n}:(\mathbb{R}^{m})^{n}\rightarrow M(\cH)$ of the covariance matrix of $\bX$ in the model $M(\cH)$ if and only 
$$n\,\,\geq\,\, \max_{i\in [m]}|\downarrow\! i|.$$ \end{thm}

Our next result is the formula for the maximal invariant (see \cite[Section 6.2]{lehmannromano}, \cite[Section 1.3]{DrKuZw2013}). It
uses the equivalence relation $\sim$ on $[m]$ defined by $i \sim j$
if and only if $N^{*}(i)=N^{*}(j)$. We write $\bar{i}$ for
the equivalence class of $i \in [m]$ and $[m]/\sim$ for the set of all
equivalence classes. 
\begin{thm} \label{thm:maxinv}
Let $\cH$ be a chain graph without flags.  Suppose that $n\geq
\max_i|\down i|$. Then the map $\tau:\RR^{m\times n}\to
\prod_{\bar{i}\in [m]/\sim} \RR^{n\times n}$ given by
\[ 
\bx \mapsto
	\left(\bx[\downb i]^T (\bx[\downb i] \bx[\downb i]^T)^{-1} 
	\bx[\downb i]\right)_{\bar{i}\in [m]/\sim},
\]
where $\bx[\downb i] \in \RR^{|\downb \; i|\times n}$ is the submatrix
of $\bx$ given by all rows indexed by\; $\down i$, is a maximal
$G^0$-invariant.
\end{thm}


\begin{exmp}Consider the model defined by $\overset{1}{\bullet}\to\overset{2}{\bullet}\ot \overset{3}{\bullet}$. Then $\downb 1=\{1\}$, $\downb 3=\{3\}$ and $\downb 2=\{1,2,3\}$. This graph is essential and the corresponding  maximal invariant statistic is
$$
(\bx[1]^{T}(\bx[1]\bx[1]^{T})^{-1}\bx[1],\;\;\bx^{T}(\bx\bx^{T})^{-1}\bx,\;\;\bx[3]^{T}(\bx[3]\bx[3]^{T})^{-1}\bx[3]),
$$
where $\bx\in \R^{3\times n}$ is a matrix whose columns are data points, and $\bx[i]$ denotes the $i$-th row of this matrix. Here $\frac{1}{n}\bx[i]\bx[i]^{T}$ is just the sample variance of $X_{i}$.
\end{exmp}

In \cite{DrKuZw2013} we also used the structure of the group to provide non-trivial bounds on the finite sample breakdown point for all equivariant estimators of the covariance matrix for undirected Gaussian graphical models. These results extends to chain graphs without flags.

\begin{prop} \label{prop:robustness}
Assume that $n \geq \max_i |\down i|$. Then for any
$G$-equi\-va\-riant
estimator $T:\RR^{m \times n} \to \SplusG$ the finite sample breakdown
point at a generic sample $\bx$ is at most $\lceil (n-\max_i |\down i|+1)/2\rceil /n$.
\end{prop}


Unlike the proof of Theorem \ref{th:main}, the proofs of Theorem \ref{th:Geqexist},  Theorem \ref{thm:maxinv} and Proposition \ref{prop:robustness} are similar to the undirected case because they depend on $G$ only through the induced poset defined by the ordering relation $N^{*}(i)\subseteq N^{*}(j)$, which drives the zero pattern of the group $G^0$. The proofs of these three results will be therefore omitted, see \cite{DrKuZw2013} for details.

\subsection*{Organization of the paper}
In Section \ref{sec:graph} we provide some basic graph-theoretical definitions. The theory of Markov equivalence of chain graphs will be discussed in Section \ref{sec:equivalence}.  In Section \ref{sec:subdet} we provide new results that give necessary and sufficient vanishing conditions for  subdeterminants of the concentration matrix $K\in \cK(\cH)$. In Section \ref{sec:GG} we analyze the structure of the group $G$ in order to prove Theorem \ref{th:main}. In Section \ref{sec:DAG} we show that in the case of DAG models, structural imsets give us all the required information to identify $G$ without constructing the essential graphs.  Section \ref{sec:examples} contains some simple examples of Theorem \ref{th:main}.

\section{Preliminaries}

In this section we discuss basic notions of the theory of chain graphs and chain graph models. 

\subsection{Basics of chain graphs}\label{sec:graph}

Let $\cH$ be a \textit{hybrid graph}, that is a graph with both directed and undirected edges, but neither loops nor multiple edges. This excludes also a situation when two nodes are connected by an undirected and a directed edge. We assume that the set of nodes of $\cH$ is labelled with $[m]=\{1,\ldots,m\}$. A directed edge (arrow) from $i$ to $j$ is denoted by $i\rightarrow j$  and an undirected edge between $i$ and $j$ is denoted by $i-j$. We write $i\cdots j$, and say that $i$ and $j$ are \textit{linked}, whenever we mean that either $i\rightarrow j$ or $i\leftarrow j$, or $i-j$.

An \textit{undirected path} between $i$ and $j$ in a hybrid graph $\cH$ is any sequence $k_1,\ldots,k_n$ of nodes such that $k_1=i$, $k_n=j$ and $k_i-k_{i+1}$ in $\cH$ for every $i=1,\ldots, n-1$. A \textit{semi-directed path} between $i$ and $j$ is any sequence $k_1,\ldots,k_n$ of nodes such that $k_1=i$, $k_n=j$ and either $k_i-k_{i+1}$ or $k_i\rightarrow k_{i+1}$ in $\cH$ for every $i=1,\ldots, n-1$ and $k_i\rightarrow k_{i+1}$ for at least one $i$. A \textit{directed path} between $i$ and $j$ in a hybrid graph $\cH$ is any sequence $k_1,\ldots,k_n$ of nodes such that $k_1=i$, $k_n=j$ and $k_i\rightarrow k_{i+1}$ in $\cH$ for every $i=1,\ldots, n-1$. A \textit{semi-directed cycle} in a hybrid graph $\cH$ is a sequence $k_1,\ldots,k_{n+1}= k_1$, $n\geq 3$ of nodes in $\cH$ such that $k_1,\ldots,k_n$ are distinct, and this sequence forms a semi-directed path. In a similar way we define a \textit{undirected cycle} and \text{directed cycle}. 

\begin{defn}
A \textit{chain graph} (or CG) is a hybrid graph without semi-directed cycles.
\end{defn}

A set of nodes $T$ is \textit{connected} in $\cH$, if for every $i,j\in T$ there exists an \emph{undirected} path between $i$ and $j$. Maximal connected subsets in $\cH$ with respect to set inclusion are called \textit{components} in $\cH$. The class of components of $\cH$ is denoted by $\sT(\cH)$. The elements of $\sT(\cH)$ form a partition of the set of nodes of $\cH$. For any subset $A\subseteq [m]$ of the set of vertices we define the \textit{induced graph} on $A$, denoted by $\cH_A$, as the graph with  set of nodes $A$ and for any two $i,j\in A$ we have $i\rightarrow j$, $j\rightarrow i$ or $i-j$ if and only if $i\rightarrow j$, $j\rightarrow i$ or $i-j$ in $\cH$, respectively.

Define the set of \emph{parents} of $A\subseteq [m]$, denoted by $p_\cH(A)$, as the set of $i\in [m]$ such that $i\rightarrow a$ in $\cH$ for some $a\in A$. The set of \emph{children} $c_\cH(A)$ is the set of $i\in [m]$ such that $a\rightarrow i$ in $\cH$ for some $a\in A$; and the set of neighbors  $n_\cH(A)$ is  the set of all $i\in [m]$ such that $i-a$ in $\cH$ for some $a\in A$. In addition we define
$$
N_{\cH}(i)\;\;:=\;\;\{i\}\cup n_{\cH}(i)\cup c_{\cH}(i).
$$

If $C$ is a connected set in a chain graph $\cH$, then there are no arrows between elements in $C$, for otherwise there would exist a semi-directed cycle. In particular, the induced graph $\cH_C$ on $C$ is an undirected graph and $p_\cH(C)$ is disjoint from $C$ for any $C\in \sT(\cH)$.  In addition, for every $A\subseteq [m]$ the induced subgraph $\cH_A$ of a chain graph $\cH$ is a chain graph itself. A \textit{clique} in an undirected graph is a subset of nodes such that any two nodes are linked. We say that a clique is \emph{maximal} if it is maximal with respect to inclusion. 
\begin{defn}
For any CG $\cH$ an \textit{immorality} is any induced subgraph of $\cH$ of the form $i\rightarrow j\leftarrow k$. A \textit{flag} is any induced subgraph of the form $i\rightarrow j-k$. A chain graph without flags is abbreviated by NF-CG.\end{defn}

 Undirected graphs and DAGs are chain graphs without flags.  We often use the following basic fact.
\begin{lem}\label{fac:nfs}
If $\cH$ is a NF-CG then $p_\cH(A)=p_\cH(T)$ for every $T\in \sT(\cH)$ and non-empty $A\subseteq T$. In particular for any two $i,j\in [m]$ such that $i-j$ in $\cH$ we have $p_\cH(i)=p_\cH(j)$.
\end{lem}

\begin{defn}
Let $\cH$ be a chain graph. For any two distinct components $T,T'\in \sT(\cH)$ consider the set of all arrows between $T$ and $T'$. If this set is non-empty then we call it a \textit{meta-arrow} and denote by $T\Rightarrow T'$. That is 
$$T\Rightarrow T':=\{i\rightarrow j:\, i\in T, j\in T', i\rightarrow j\mbox{ in }\cH\}.$$
\end{defn}
The notion of meta-arrow is important in the considerations of equivalence classes of chain graphs, which we discuss in the next section.

\subsection{Equivalence classes of chain graphs}\label{sec:equivalence} 

A chain graph model is given by all concentration matrices of the form (\ref{eq:parameter}). In Example \ref{ex:DAG3} we saw that two different chain graphs may give the same Gaussian models or equivalently the same set of conditional independence statements. If two NF-CGs $\cG$ and $\cH$ define the same chain graph model, we say that they are \textit{graph equivalent} (or simply \textit{equivalent}). For example the three DAGs in Figure \ref{fig:3DAGs} are equivalent. 
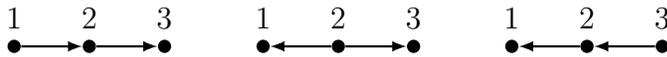
\begin{figure}[htp!]
\begin{center}
\tikzstyle{vertex}=[circle,fill=black,minimum size=5pt,inner sep=0pt]
  \begin{tikzpicture}
    \node[vertex] (1) at (0,0)  [label=above:$1$] {};
    \node[vertex] (2) at (1,0) [label=above:$2$]{};
    \node[vertex] (3) at (2,0) [label=above:$3$]{};
          \draw[->,-latex,line width=.3mm] (1) to (2);
    \draw[->,-latex,line width=.3mm] (2) to (3);
  \end{tikzpicture}\qquad
  \begin{tikzpicture}
    \node[vertex] (1) at (0,0)  [label=above:$1$] {};
    \node[vertex] (2) at (1,0) [label=above:$2$]{};
    \node[vertex] (3) at (2,0) [label=above:$3$]{};
          \draw[<-,-latex,line width=.3mm] (2) to (1);
    \draw[->,-latex,line width=.3mm] (2) to (3);
  \end{tikzpicture}\qquad
  \begin{tikzpicture}
    \node[vertex] (1) at (0,0)  [label=above:$1$] {};
    \node[vertex] (2) at (1,0) [label=above:$2$]{};
    \node[vertex] (3) at (2,0) [label=above:$3$]{};
          \draw[->,-latex,line width=.3mm] (2) to (1);
    \draw[->,-latex,line width=.3mm] (3) to (2);
  \end{tikzpicture}
  \end{center}
  \caption{Three equivalent DAGs. }\label{fig:3DAGs}
\end{figure}

The equivalence class of $\cH$ in the set of NF-CGs is denoted by $\la\cH\ra$: 
$$
\la\cH\ra=\{\cG:\,\cG\mbox{ is a NF-CG graph equivalent to } \cH\}.
$$
Equivalence of CGs and DAGs was discussed in many papers, for example \cite{andersson1997,frydenberg1990,roverato2005,vermapearl91}. We briefly list the most relevant results.
\begin{defn}
The \textit{skeleton} of a chain graph  $\cH$ is the undirected graph such that $i-j$ whenever $i\cdots j$ in $\cH$. 
\end{defn}
\begin{thm}\label{th:frydenberg}Two NF-CGs with the same set of nodes are equivalent if and only if they have the same skeleton and the same immoralities.
\end{thm}
The original statement of this result, given by Frydenberg in \cite{frydenberg1990}, is more general and applies to any chain graph in the  LWF definition of chain graph models.  

As was remarked in \cite{roverato2005} considering meta-arrows helps to understand equivalence classes of chain graphs. Suppose that we want to obtain one chain graph from another with the same skeleton by changing some of the arrows $i\rightarrow j$ to $i-j$ or $i\leftarrow j$. Changing only a subset  of arrows in a meta-arrow $T\Rightarrow T'$ is not permitted as it would introduce semi-directed cycles. Hence the only permitted operations on arrows of $\cH$, if we work in the class of CGs, is either changing the directions of all the elements of $T\Rightarrow T'$ or changing all arrows of $T\Rightarrow T'$ into undirected edges. The following basic operation on a chain graph was defined in \cite{roverato2005,studeny2004}.
\begin{defn}\label{def:merging}Let $\cH$ be a NF-CG and let $T\Rightarrow T'$ be a meta-arrow in $\cH$ where $T,T'\in \sT(\cH)$. \textit{Merging} of $T$ and $T'$ is an operation of changing all elements of the meta-arrow $T\Rightarrow T'$ into undirected edges. Merging is called \textit{legal} if 
\begin{itemize}
\item[(a)] $p_{\cH}(T')\cap T$ is a clique of $T$;
\item[(b)] $p_{\cH}(T')\setminus T=  p_{\cH}(T)$.
\end{itemize}
\end{defn}

\begin{lem}Let $\cH$ be a NF-CG and let $\cH'$ be a graph obtained from $\cH$ by legal merging of two connected components. Then $\cH'\in \la\cH\ra$. 
\end{lem}
\begin{proof}See for example the proof of Lemma 22 in \cite{studeny2009}.
\end{proof}

For two distinct CGs $\cG$, $\cH$ with the same skeleton we write $\cG\subseteq \cH$ if, whenever $i\rightarrow j$ in $\cG$, then either $i\rightarrow j$ or $i-j$ in $\cH$, and whenever $i-j$ in $\cG$, then $i-j$ in $\cH$. We write $\cG\subset \cH$ if $\cG\subseteq \cH$ and $\cG\neq \cH$.

\begin{thm}[Roverato,Studeny \cite{roverato2005,studeny2004}]\label{th:roverato} Let $\cG$ and $\cH$ be two equivalent NF-CGs such that $\cG\subset \cH$. Then there exists a finite sequence $\cG=\cG_0\subset \cdots \subset \cG_r=\cH$, with $r\geq 1$, of equivalent NF-CGs such that, for all $i=1,\ldots, r$ $\cG_i$ can be obtained from $\cG_{i-1}$ by a legal merging of two connected components of $\cG_{i-1}$. \end{thm}

By the following proposition there is always a unique NF-CG representing $\la \cH\ra$ with the largest number of undirected edges. 
\begin{prop}[Roverato,Studeny \cite{roverato2005,studeny2004}]\label{prop:essexists} There exists a unique  element $\cH^*$ in $\la\cH\ra$ that is maximal in the sense that  $\cH'\subseteq \cH^*$ for every $\cH'\in \la\cH\ra$. 
\end{prop}

\begin{defn}\label{def:essential}Let $\cH$ be a NF-CG. The graph $\cH^*$ of Proposition \ref{prop:essexists} is called the \textit{essential graph}. The directed arrows in $\cH^{*}$ are called \textit{essential}. For notational convenience we write $p^*(A)$, $n^*(A)$ and $c^*(A)$ for $p_{\cH^*}(A)$, $n_{\cH^*}(A)$ and $c_{\cH^*}(A)$ respectively. 
\end{defn}

By definition $\cH^*$ has the same skeleton as $\cH$, and an edge is essential if and only if it occurs as an arrow with the same orientation in every $\cH'\in \la\cH\ra$; all other edges are undirected. For example, the essential graph for any of the graphs in Figure \ref{fig:3DAGs} is the undirected graph $\overset{1}{\bullet}-\overset{2}{\bullet}-\overset{3}{\bullet}$, whereas the essential graph of $\cH=\overset{1}{\bullet}\rightarrow \overset{2}{\bullet}\leftarrow \overset{3}{\bullet}$ is $\cH$ itself.  By Theorem \ref{th:frydenberg}, every arrow that participates in an immorality in $\cH$ is essential, but $\cH$ may contain other essential arrows. For example, in the DAG in Figure \ref{fig:DAGallessential} all arrows are essential but not all of them form immoralities.

\begin{figure}[htp!]
\begin{tikzpicture}
\tikzstyle{vertex}=[circle,fill=black,minimum size=5pt,inner sep=0pt]
    \node[vertex] (1) at (0,0)   [label=below:$1$]{};
    \node[vertex] (2) at (1,0) [label=below:$3$]{};
    \node[vertex] (3) at (.5,1) [label=above:$2$]{};
    \node[vertex] (4) at (1.5,1) [label=above:$4$]{};    
    \draw[->,-latex,line width=.3mm] (1) to (2);
    \draw[->,-latex,line width=.3mm] (3) to (2);
    \draw[->,-latex,line width=.3mm] (1) to (3);
    \draw[->,-latex,line width=.3mm] (4) to (3);
  \end{tikzpicture}
  \caption{A NF-CG whose arrows are all essential but not all part of  immoralities.}\label{fig:DAGallessential}
  \end{figure}
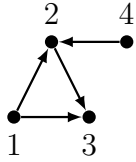

The following result has been independently observed in \cite{roverato2005,studeny2004}.
\begin{thm}
If $\la \cH\ra$ contains a DAG $\cG$, then the essential graph $\cH^*$ is equal to the essential graph of a DAG as defined in \cite{andersson1997}. \end{thm}

\begin{rem}\label{rem:essential}
Our terminology is consistent with \cite{roverato2005}. However,  in  \cite{andersson2006} the essential graph for a chain graph is defined in a different way and it corresponds to the essential graph $\cH^*$ only if $\la\cH\ra$ contains a DAG. \end{rem}

\section{Subdeterminants of concentration matrices}\label{sec:subdet}

Let $\cH$ be any chain graph on $[m]$. We want to determine which
sub-determinants of the concentration matrix of the corresponding model
are identically zero on the model. This provides simple necessary conditions for a concentration matrix to lie in $\cK(\cH)$. We will use the following
combinatorial notions.

\begin{defn}
A {\em cup} in $\cH$ is a quadruple $(i,j,k,l)$ of vertices in $\cH$ where
\begin{enumerate}
\item either $i=j$ or $i \to j$; {\em and}
\item either $j=k$ or $j-k$; {\em and}
\item either $k=l$ or $k \ot l$.
\end{enumerate}
We say that the cup {\em starts} in $i$ and {\em ends} in $l$. 
\end{defn}

\begin{defn}
Let $A$ and $B$ be sets of vertices of $\cH$ of the same cardinality
$d$. A {\em cup system} from $A$ to $B$ is a set $U$ of $d$ cups in
$\cH$ whose starting points exhaust $A$ and whose end points exhaust
$B$.  The cup system $U$ from $A$ to $B$ gives rise to a bijection
$A \to B$ that sends $a \in A$ to the end point of the cup in $U$
that starts with $a$. After fixing labellings $A=\{a_1,\ldots,a_d\}$
and $B=\{b_1,\ldots,b_d\}$ this bijection gives rise to a permutation
of $[d]$; define $\sgn(U)$ to be the sign of this permutation.
The cup system $U$ from $A$ to $B$ is said to be {\em self-avoiding} if
for each $k =1,2,3,4$ the elements $u_k \in [m]$ of $u=(u_{1},u_{2},u_{3},u_{4}) \in U$ are all distinct.
\end{defn}

For the graph $\overset{1}{\bullet}\to \overset{2}{\bullet}\ot \overset{3}{\bullet}$ there is no self-avoiding cup system from $\{1,2\}$ to $\{2,3\}$ but there is such a system between $\{1\}$ and $\{3\}$. 

\begin{defn}
Let $\lambda_{ij}$ be the parameters corresponding to arrows $i \to j$ in
$\cH$ and let $\omega_{ij}$ be the parameters corresponding to undirected
edges $i-j$ and to the diagonal ($\omega_{ii}$). The {\em weight} of a cup
$(i,j,k,l)$ in $\cH$ is the product of the $(i,j)$ entry of $(I-\Lambda)$,
the $(j,k)$-entry of $\Omega$, and the $(k,l)$-entry of $(I-\Lambda)^T$,
which is the $(l,k)$-entry of $(I-\Lambda)$. The weight of a cup system
$U$ from $A$ to $B$, denoted $w(U)$, is the product of the weights of
the cups in $U$. This is a monomial of degree $k$ in the $\omega_{ij}$
times a monomial of degree at most $k$ in the variables $-\lambda_{ij}$.
\end{defn}

Let $K[A,B]$ denote the $A \times B$-submatrix of
$K=(I-\Lambda)\Omega(I-\Lambda)^T$. By expanding the entries, we find that
\begin{equation}  \label{eq:Expansion}
\det K[A,B]=\sum_{U} \sgn(U) w(U), 
\end{equation}
where the sum is over all cup systems $U$ from $A$ to $B$.  In this
expression cancellation can occur because of the signs $\sgn(U)$
(not because of the signs in the $-\lambda_{ij}$, which we might as
well have taken as new variables). The following proposition captures
exactly which terms cancel. For more details on the arguments, we refer
to \cite{sullivant2008tsg,draisma2013positivity}.

\begin{prop}\label{prop:monomialsInDet}
Relative to the fixed labellings of $A$ and $B$, the $A \times
B$-subdeterminant of $K$ equals
\[ \det K[A,B]=\sum_{U \text{ self-avoiding}} \sgn(U) w(U). \]
Moreover, for any two self-avoiding cup systems $U$ and $U'$ with
$w(U)=w(U')$ we have $\sgn(U)=\sgn(U')$.
\end{prop}

\begin{proof}
To see that the sum in \eqref{eq:Expansion} can be restricted
to self-avoiding cup systems $U$, we proceed as in the
Lindstr\"om-Gessel-Viennot lemma \cite[Theorem 1]{gessel89} and give a sign-reversing
involution $\sigma$ on the set of non-self-avoiding cup systems, as
follows. Order any cup system $U$ from $A$ to $B$ as $\{u_1,\ldots,u_d\}$
where $u_i$ starts in $a_i$.  If $U$ is not self-avoiding, let $a \in
\{2,3\}$ be minimal such that the entries $u_{ia}, i \in [d]$ are not
all distinct, and let $(i,i')$ be a lexicographically minimal pair such
that $u_{ia}=u_{i'a}$. Then $\sigma(U)$ is the cup system obtained from
$U$ by replacing $u_i$ and $u_{i'}$ by their swaps at position $a$. For
instance, if $a=2$, then $u'_i=(u_{i1},u_{i2}=u_{i'2},u_{i'3},u_{i'4})$
and $u'_{i'}=(u_{i'1},u_{i'2}=u_{i2},u_{i3},u_{i4})$; and similarly for
$a=3$. Now $\sgn(U')=-\sgn(U)$ and $\sigma$ is indeed an involution.
This proves the expression in the proposition.  The second statement is
more subtle, but it follows by applying \cite[Theorem 3.3]{draisma2013positivity} 
to the DAG obtained from $\cH$ by reversing all arrows and replacing
all undirected edges $i-j$ by a pair $i \ot k \to j$ of arrows,
where $k$ is a new vertex. Indeed, self-avoiding cup systems in $\cH$
correspond to special types of trek systems without sided intersection
in that new graph. 
\end{proof}

Note that the set of covariance matrices in the model is captured by which subdeterminants vanish identically --- indeed, the conditional independence statements already suffice for this, and they are determinants (see for example \cite[Proposition 3.1.13]{oberwolfach2009}) --- but we do not know if this is true for the set of concentration matrices as well. Therefore, Proposition \ref{prop:monomialsInDet} may well have other statistical applications, but in
what follows, we will mostly use the following direct consequence.

\begin{cor} \label{cor:AB}
The subdeterminant $\det K[A,B]$ is identically zero on the model
corresponding to $\cH$ if and only if there does not exist a self-avoiding
cup system from $A$ to $B$ in $\cH$.
\end{cor}

In the next section we begin our analysis of the group $G$, defined in (\ref{eq:mainproblem}), with a study of its connected component of the identity.

\section{The group $G$}\label{sec:GG}

\subsection{The connected component of the identity}

Denote by $E_{ij}$ the matrix in $\R^{m\times m}$ with all entries zero apart from the $(i,j)$-th element which is $1$. By $G^0$ denote the normal subgroup of $G$ which forms the connected component of the identity matrix. The subgroup $T^{m}$ of all diagonal and invertible matrices is contained in the group $G$ because scaling of vector $\bX$ does not affect conditional independencies. By \cite[Lemma 2.1]{DrKuZw2013}, to compute $G^{0}$, it suffices to check for which $(i,j)\in [m]\times [m]$ the one-parameter groups $(I+tE_{ij})$, $t\in \R$, lie in $G$; or equivalently $E_{ij}\in \liea$, where $\liea$ is the Lie algebra of $G$. 

Before we provide the main result of this section we recall \cite[Proposition 2.2]{DrKuZw2013}.
\begin{prop}[Proposition 2.2,  \cite{DrKuZw2013}]\label{prop:H0gens}
Let $\cH$ be an undirected graph. For $i, j\in [m]$ the matrix $E_{ij}$ lies in $\liea$ if and only if $N_{\cH}(i)\subseteq N_{\cH}(j)$. 
\end{prop}
If $\cH$ is a NF-CG such that $\cH^{*}$ is an undirected graph then Proposition \ref{prop:H0gens} can be used to characterize $G^{0}$ for $\cH$ by passing to the essential graph.  However, it is not immediately clear how this result extends to all chain graphs without flags. We first note that one direction of the above result holds in general. 

\begin{lem} \label{lm:G0lower}
Let $\cH$ be an NF-CG. If $N_{\cH}(i)\subseteq N_{\cH}(j)$, then $E_{ij}\in\liea$.\end{lem}

\begin{proof}
If $i=j$ then the statement is clear so suppose that $i\neq j$. We have $N_{\cH}(i)\subseteq N_{\cH}(j)$ only if either $j\to i$ or $i-j$ in $\cH$. Suppose first that $j\to i$. We have
$$
(I-tE_{ji})(I-\Lambda)\Omega(I-\Lambda)^{T}(I-tE_{ij})=(I-\tilde\Lambda)\Omega(I-\tilde\Lambda)^{T},
$$
where $\tilde\Lambda=-\Lambda-tE_{ji}+tE_{ji}\Lambda$; $\tilde\lambda_{uv}=\lambda_{uv}$ if $u\neq j$; $\tilde\lambda_{jv}=\lambda_{jv}-t\lambda_{iv}$ if $v\neq i$; and $\tilde\lambda_{ji}=\lambda_{ji}+t$. The fact that $\tilde\Lambda$ lies in $\R^{\cH}$ follows from $c_{\cH}(i)\subseteq c_{\cH}(j)$ and hence for every $v$ if $\lambda_{jv}= 0$ then $\lambda_{iv}=0$.

\medskip
If $i-j$ in $\cH$ then $i\cup n_{\cH}(i)\subseteq j\cup n_{\cH}(j)$ and $p_{\cH}(i)=p_{\cH}(j)$ by Lemma \ref{fac:nfs}. By Proposition \ref{prop:H0gens} applied to the undirected part of $\cH$ we can write $\Omega=(I+tE_{ji})\tilde\Omega(I+t E_{ij})$ for some $\tilde\Omega\in \cS_{\cH}^{+}$. Therefore 
$$
(I-tE_{ji})(I-\Lambda)\Omega(I-\Lambda)^{T}(I-tE_{ij})=(I-tE_{ji})(I-\Lambda)(I+tE_{ji})\tilde\Omega(I+tE_{ij})(I-\Lambda)^{T}(I-tE_{ij}),
$$
where we now show that there exists $\tilde\Lambda\in \R^{\cH}$ such that
$$
(I-tE_{ji})(I-\Lambda)(I+tE_{ji})=(I-\tilde\Lambda).
$$
Indeed, 
$$
\tilde\Lambda=\Lambda+t\Lambda E_{ji}-tE_{ji}\Lambda +t^{2}E_{ji}\Lambda E_{ji}
$$
where the last term must vanish because $\lambda_{ij}=0$. Hence $\tilde\Lambda$ is obtained from $\Lambda$ by adding a multiple of the $j$-th column to the $i$-th column and by adding a multiple of the $i$-th row to the $j$-th row. The fact that $\tilde\Lambda$ lies in $\R^{\cH}$ follows from the fact that $c_{\cH}(i)\subseteq c_{\cH}(j)$ and $p_{\cH}(i)= p_{\cH}(j)$, that is, the $i$-th column has the same support as the $j$-th column and  the support of the $i$-th row is contained in the support of the $j$-th row.
\end{proof}

The converse of the lemma does not hold for general NF-CG $\cH$. Consider for instance $\overset{1}{\bullet}\to\overset{2}{\bullet}\to\overset{3}{\bullet}$. By Example \ref{ex:DAG3}, the element $I+tE_{12}$ lies in $G^{0}$ but $\{1,2\}\not\subseteq \{2,3\}$. Nevertheless, the converse of the lemma above does hold when $\cH$ is essential; this is the main result of
this section. 

\begin{thm} \label{thm:G0}
Let $\cH$ be an essential NF-CG. Then  $E_{ij}\in \liea$ if and only if  $N_{\cH}(i)\subseteq N_{\cH}(j)$.
\end{thm}
The proof is moved to the Appendix.

As we noted in the beginning of this section, the set of all $E_{ij}\in \liea$ gives already the complete information on the group $G^{0}$. Hence Theorem~\ref{thm:G0} gives the description of $G^{0}$ in (\ref{eq:G0}).

\begin{exmp}Consider a DAG $\cH=\overset{1}{\bullet}\rightarrow \overset{2}{\bullet}\leftarrow\overset{3}{\bullet}$ then $N_{\cH}(2)\subseteq N_{\cH}(1)\cap N_{\cH}(3)$. Hence both $E_{21}$ and $E_{23}$ lie in $\liea$ but no other off-diagonal elements of matrices in $G^0$ can be non-zero.
\end{exmp}

\subsection{The component group}\label{sec:fullG}

Note that $G^0$ given in Theorem~\ref{thm:G0} in general is not the whole group $G$. For example both for the model $\overset{1}{\bullet}\rightarrow\overset{2}{\bullet}\leftarrow \overset{3}{\bullet}$ and for any of the equivalent DAGs in Figure \ref{fig:3DAGs} the permutation matrix
$$
\left[\begin{array}{ccc}
0 & 0 & 1\\
0 & 1 & 0\\
1 & 0 & 0
\end{array}\right]
$$
lies in $G$ but not in $G^0$. The following result shows that permutation matrices form the basis for understanding the remaining part of the group $G$. For the proof see \cite[Proposition 2.5]{DrKuZw2013}.
\begin{prop}\label{lem:asperm}
Every element $g\in G$ can be written as $g=\sigma g_0$, where $g_0\in
G^0$ and $\sigma$ is a permutation matrix contained in $G$.
\end{prop}

An automorphism of a hybrid graph is any bijection $\sigma:\, [m]\rightarrow [m]$ of its nodes such that for every $i,j\in [m]$ we have $\sigma(i)-\sigma(j)$ if and only if $i-j$ and $\sigma(i)\rightarrow \sigma(j)$ if and only if $i\rightarrow j$. 
\begin{lem}\label{lem:sigmainG}
Let $\cH$ be a NF-CG and $\cH^*$ its essential graph. Let $\sigma\in {\rm GL}_m(\R)$ be a permutation matrix. Then $\sigma\in G$ if and only if $\sigma$ is an automorphism of $\cH^*$.
\end{lem}
\begin{proof}
The model $M(\cH)$ is uniquely defined by the set of conditional independence statements (see for example \cite{lauritzen:96}). Given a set of such statements that come from a chain graph $\cH$ the equivalence class  $\la\cH\ra$ is determined uniquely. The essential graph $\cH^*$ is the unique representative of $\la\cH\ra$ with the largest number of undirected edges. Since any permutation $\sigma$ applied to $\cH^*$ gives a NF-CG with the same number of undirected and directed edges (it simply relabels the nodes), $\sigma$ lies in the model if and only if $\sigma$ is an automorphism of $\cH^*$.
\end{proof}

By Lemma \ref{lem:sigmainG} we can conclude that $G$ is generated by $G^0$ and the automorphism group of $\cH^*$, which proves Theorem \ref{th:main}.

Define an equivalence relation on $[m]$ by $i\sim j$ whenever $N^{*}(i)=N^{*}(j)$. For example if $\cH=\overset{1}{\bullet}\rightarrow\overset{2}{\bullet}$ then $\cH^*=\overset{1}{\bullet}-\overset{2}{\bullet}$ and hence $1\sim 2$. The equivalence class of $i\in [m]$ is denoted by $\bar{i}$.

As explained in the introduction, the expression $G=\Aut(\cH^*) G^0$
is not minimal in the sense that $\Aut(\cH^*)$ and $G^0$ may intersect.
To get rid of that intersection, we define $\tcH^*$ to be the graph with
vertex set $[m]/\sim$ and $\bar{i}\rightarrow\bar{j}$ ($\bar{i}-\bar{j}$) in $\widetilde{\cH}^*$ if and only if $i\rightarrow j$ ($i-j$) in $\cH$. We first show that $\widetilde{\cH}^*$ is well defined. \begin{lem}\label{lem:equiv}Let $\cH$ be a NF-CG and $\cH^*$ its essential graph. Two elements $i,j\in [m]$ are equivalent if and only if $\{i\}\cup n^*(i)=\{j\}\cup  n^*(j)$, $p^*(i)= p^*(j)$ and $c^*(i)= c^*(j)$. In particular the graph $\widetilde{\cH}^*$ is well-defined.
\end{lem} 
\begin{proof}If $N^{*}(i)=N^{*}(j)$ then $i$ and $j$ are necessarily linked. Since $i\in N^{*}(j)$ and $j\in N^{*}(j)$ we conclude that in fact $i-j$ in $\cH^{*}$. By Lemma \ref{fac:nfs}, since $i-j$, we also have $p^*(i)=p^*(j)$. This shows that $i\sim j$ if and only if $\{i\}\cup n^*(i)=\{j\}\cup n^*(j)$, $c^*(i)=c^*(j)$ and $p^*(i)=p^*(j)$, which shows that the definition of the arrows and edges in $\cH$ is independent of the representative $i$ and $j$. \end{proof}

Define $c:[m]/\sim \to \NN,\
\bar{i} \mapsto |\bar{i}|$ and view $c$ as a coloring of the vertices
of $\tcH^*$ by natural numbers. Let $\Aut(\tcH^*,c)$ denote the group of
automorphisms of $\tcH^*$ preserving the coloring. There is a lifting
$\ell: \Aut(\tcH^*,c) \to \Aut(\cH^*)$ defined as follows: the element $\tau
\in \Aut(\tcH^*,c)$ is mapped to the unique bijection $\ell(\tau): [m]
\to [m]$ that maps each equivalence class $\bar{i}$ to the equivalence
class $\tau(\bar{i})$ by sending the $k$-th smallest element of $\bar{i}$
(in the natural linear order on $[m]$) to the $k$-th smallest element
of $\tau(\bar{i})$, for $k=1,\ldots,|\bar{i}|$.

 \begin{exmp} Consider a DAG $\cH$ and its essential graph $\cH^*$ in Figure \ref{fig:4exambis}. Since $3$ and $4$ are equivalent, the induced essential  graph $\widetilde{\cH}^*$ is equal to $\overset{1}{\bullet}-\overset{2}{\bullet}-\!\overset{3,4}{\bullet}$. There are no non-trivial automorphisms of this graph preserving cardinality of equivalence classes and ${\rm Aut}(\widetilde{\cH}^*,c)=\{I\}$. In particular $\ell$ is a trivial mapping.
 \end{exmp}
  \begin{figure}[htp!]
\begin{tikzpicture}
\tikzstyle{vertex}=[circle,fill=black,minimum size=5pt,inner sep=0pt]
    \node[vertex] (1) at (0,0)  [label=above:$2$] {};
    \node[vertex] (4) at (-1.2,0) [label=left:$1$]{};
    \node[vertex] (2) at (1,.8) [label=above:$3$]{};
    \node[vertex] (3) at (1,-.8) [label=below:$4$]{};
    \draw[->,-latex,line width=.3mm] (1) to (2);
    \draw[->,-latex,line width=.3mm] (1) to (3);
    \draw[->,-latex,line width=.3mm] (1) to (4);
        \draw[->,-latex,line width=.3mm] (2) to (3);
  \end{tikzpicture}\qquad \qquad \begin{tikzpicture}
\tikzstyle{vertex}=[circle,fill=black,minimum size=5pt,inner sep=0pt]
    \node[vertex] (1) at (0,0)  [label=above:$2$] {};
    \node[vertex] (4) at (-1.2,0) [label=left:$1$]{};
    \node[vertex] (2) at (1,.8) [label=above:$3$]{};
    \node[vertex] (3) at (1,-.8) [label=below:$4$]{};
    \draw[line width=.3mm] (1) to (2);
    \draw[line width=.3mm] (1) to (3);
    \draw[line width=.3mm] (1) to (4);
        \draw[line width=.3mm] (2) to (3);
  \end{tikzpicture}

  \caption{On the left a DAG on four nodes. On the right its essential graph.}\label{fig:4exambis}
  \end{figure}
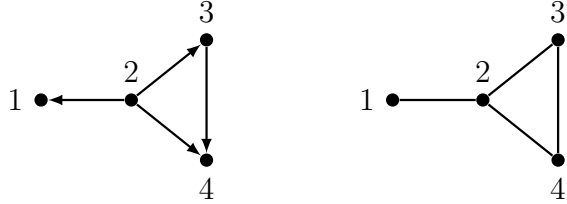

\begin{thm}\label{th:main2}
The group $G$ equals $\ell(\Aut(\tcH^*,c)) G^0$, and the intersection
$\ell(\Aut(\tcH^*,c)) \cap G^0$ is trivial, so $G$ is the semidirect product
$\ell(\Aut(\tcH^*,c)) \ltimes G^0$. 
\end{thm}
\begin{proof}
It is a standard result from the Lie group theory that the connected component of the identity $G^{0}$ is a normal subgroup of $G$. Hence, to show that $G=G^{0}\rtimes {\rm Aut}(\widetilde{\cH}^*,c)$ we need to show that $G=G^{0}\cdot {\rm Aut}(\widetilde{\cH}^*,c)$ and $G^{0}\cap {\rm Aut}(\widetilde{\cH}^*,c)=\{I\}$. The first part follows by Proposition \ref{lem:asperm} and  Lemma \ref{lem:sigmainG}. To show that $G^{0}\cap {\rm Aut}(\widetilde{\cH}^*,c)=\{I\}$ note that transpositions of $i$ and $j$ lie in $G^{0}$ precisely when $i$ and $j$ are equivalent and hence, when they do not lie in $\ell({\rm Aut}(\widetilde{\cH}^*,c))$.
%
%
\end{proof}

\begin{rem}To the coloured graph $(\widetilde\cH^*,c)$ we can associate a Gaussian graphical model $M(\cH,c)$ with \textit{multivariate nodes}, where node $\bar i$ is associated to a Gaussian vector of dimension ${c_{\bar i}}$. This model coincides with $M(\cH)$. This also shows, conversely, that our framework extends to general Gaussian graphical models of chain graphs with no flags with multivariate nodes.
\end{rem}

Computing the essential graph $\cH^*$ is not always a simple task. In Section \ref{sec:DAG} we show how to identify the group $G$ without finding $\cH^*$ in the case when $\cH$ is a DAG. In the next section we illustrate Theorem \ref{th:main2} with some basic examples.

\section{Efficient computations for DAG models}\label{sec:DAG}

In this section we present some efficient techniques for computing the group $G^{0}$ in the case when $\cH$ is a DAG. The following characterization of essential graphs of DAGs will be useful.  
\begin{thm}[Roverato, Studen\'{y} \cite{roverato2005}\cite{studeny2004}]\label{thm:essDAG}If $\cH$ is a DAG then each connected component of $\cH^*$ is decomposable. Moreover, $\cH^*$ coincides with the essential graph of $\cH$ as defined in \cite{andersson1997} (see also Remark \ref{rem:essential}). 
\end{thm}

For any DAG $\cH$ on the set of nodes $[m]$, the \textit{standard imset} for $\cH$ is an integer-valued function $u_\cH:\, 2^{[m]}\rightarrow  \mathbb{Z}$, where $2^{[m]}$ is the set of all subsets of $[m]$, defined by
\begin{equation}\label{eq:uG}
u_\cH\quad:=\quad\delta_{[m]}-\delta_\emptyset+\sum_{i\in [m]}(\delta_{p_{\cH}(i)}-\delta_{p_{\cH}(i)\cup \{i\}}),
\end{equation}
where $\delta_A:\,2^{[m]}\rightarrow \{0,1\}$ satisfies $\delta_A(B)=1$ if $A=B$ and is zero otherwise. For example, it is easy to verify that all DAGs in Figure \ref{fig:3DAGs} give raise to the imset represented by Figure \ref{fig:imset3DAGs}.
\begin{lem}[Corollary 7.1, \cite{studeny2005pci}]\label{lem:uGequiv}
Let $\mathcal{G},\cH$ be two DAGs. Then $\cH\in \la\mathcal{G}\ra$ if and only if $u_\mathcal{G}=u_{\cH}$.
\end{lem}
\begin{figure}[htp!]
\begin{tikzpicture}[scale=0.8]
\tikzstyle{every node}=[draw,shape=circle split, inner sep=0pt,minimum size =1.2cm];
  \node (max) at (0,4) { {${}^{\{1,2,3\}}$}\nodepart{lower} \begin{minipage}[c]{0.5cm}\vspace{2pt}$+1$\end{minipage}};
  \node (a) at (-2,2) {{${}^{\{1,2\}}$}\nodepart{lower} \begin{minipage}[c]{0.6cm}\vspace{3pt}$-1$\end{minipage}};
  \node (b) at (0,2) {{${}^{\{1,3\}}$}\nodepart{lower} \begin{minipage}[c]{0.2cm}\vspace{3pt}$0$\end{minipage}};
  \node (c) at (2,2) {{${}^{\{2,3\}}$}\nodepart{lower} \begin{minipage}[c]{0.6cm}\vspace{3pt}$-1$\end{minipage}};
  \node (d) at (-2,0) {{${}^{\{1\}}$}\nodepart{lower} \begin{minipage}[c]{0.2cm}\vspace{3pt}$0$\end{minipage}};
  \node (e) at (0,0) {{${}^{\{2\}}$}\nodepart{lower} \begin{minipage}[c]{0.6cm}\vspace{3pt}$+1$\end{minipage}};
  \node (f) at (2,0) {{${}^{\{3\}}$}\nodepart{lower} \begin{minipage}[c]{0.2cm}\vspace{3pt}$0$\end{minipage}};
  \node (min) at (0,-2) {{${}^\emptyset$}\nodepart{lower} \begin{minipage}[c]{0.2cm}\vspace{3pt}$0$\end{minipage}};
  \draw (min) -- (d) -- (a) -- (max) -- (b) -- (f)
  (e) -- (min) -- (f) -- (c) -- (max)
  (d) -- (b);
  \draw[preaction={draw=white, -,line width=6pt}] (a) -- (e) -- (c);
\end{tikzpicture}
\caption{The imset $u_\cH$, where $\cH$ is any of the three equivalent DAGs in Figure \ref{fig:3DAGs}.}\label{fig:imset3DAGs}
\end{figure}
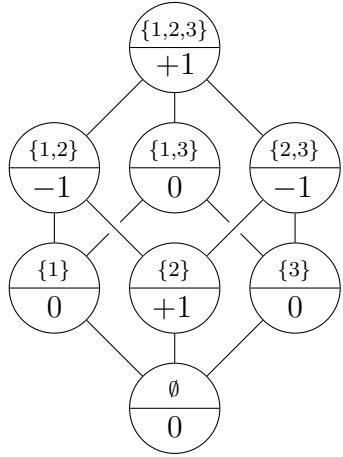

 The support of $u_\cH$ for a DAG $\cH$ has been described in \cite{studenyessential} directly in terms of the essential graph. To provide this result we introduce some useful notions related to chain graphs. 
 \begin{defn}\label{def:idle}
A set $B\subseteq[m]$ of nodes in a chain graph $\cH$ is \textit{idle} if $i\cdots j$ for all $i, j\in B$; and for every $i\in [m]\setminus B$ and every $j\in B$, $i\rightarrow j$ in $\cH$.\end{defn}
 By  \cite[Lemma 18]{studeny2009} every chain graph has a unique maximal idle set of nodes (which may be empty), which we denote by ${\rm idle}(\cH)$. The complement of the largest idle set is called the \textit{core} of $\cH$ and denoted ${\rm core}(\cH)$. Directly from the definition it follows that ${\rm idle}(\cH)$ is a union of connected components of $\cH$. Therefore, the core is also a union of connected components. The class of core-components, that is, components in $\cH$ contained in ${\rm core}(\cH)$ is denoted by $\sT_{\rm core}(\cH)$.
 

 \begin{lem}\label{lem:idle}If $\la\cH\ra$ is a NF-CF then ${\rm idle}(\cH^{*})$ forms a clique, that is, all its nodes are connected by an undirected edge.
\end{lem}
\begin{proof}Because there is a directed arrow from any node outside ${\rm idle}(\cH^{*})$ to any node in ${\rm idle}(\cH^{*})$, every component of $\cH^{*}$ lies either inside or outside of ${\rm idle}(\cH^{*})$. Since all nodes in ${\rm idle}(\cH^{*})$ are linked, there is a meta-arrow between any two distinct components of ${\rm idle}(\cH^{*})$ and each component is a clique. Without loss of generality pick $T$ such that $T'$ is the only child-component of $T$. First note that $p^{*}(T')\cap T=T$ forms a clique. Second, the parent-components of $T'$ are $T\cup p^{*}(T)$. Indeed, if a component $S$, such that $S\Rightarrow T'$, lies outside of ${\rm idle}(\cH^{*})$ then $S\subseteq p^{*}(T)$ by definition. If $S\subseteq {\rm idle}(\cH^{*})$ then $S\subseteq p^{*}(T)$ because $S$ and $T$ are necessarily linked and $T$ has no other children than $T'$. Thus, by Definition \ref{def:merging}, $T$ and $T'$ can be legally merged, which contradicts the fact that $\cH^{*}$ is essential. 
\end{proof}

Note that ${\rm idle}(\cH^{*})$ is precisely the set of vertices $i$ such that $\downb i=[m]$, where $\downb i=\{j:\,N^{*}(i)\subseteq N^{*}(j)\}$. 

From now on $\cH$ will always denote a DAG. By Theorem \ref{thm:essDAG} each component $T\in \sT_{{\rm core}}(\cH^*)$ induces a decomposable graph $\cH^*_T$. We recall that a decomposable graph is an undirected graph with no induced cycles of size $\geq 4$. An alternative definition, that will be useful in this section, is that its maximal cliques can be ordered into  a sequence $C_{1},\ldots,C_{p}$ satisfying the \emph{running intersection property} (see \cite[Proposition 2.17]{lauritzen:96}), that is
\begin{equation}\label{eq:RunningInt}
 \forall i\geq 2\;\;\; \exists k<i\qquad S_{i}=C_{i}\cap \left(\bigcup_{j<i} C_{j}\right)\subseteq C_{k}.
\end{equation}
By \cite[Lemma 7.2]{studeny2005pci} the collection of sets $S_{i}$ for $2\leq i\leq m$ does not depend on the choice of ordering that satisfies (\ref{eq:RunningInt}). We call these sets \emph{separators} of the graph. The multiplicity $\nu(S)$ of a separator $S$ is then defined as the number of indices $i$ such that $S_{i}=S$. This number also does not depend on the choice of an ordering that satisfies (\ref{eq:RunningInt}).

By $\sC(T)$ we denote the collection of maximal cliques of $\cH^*_T$, by $\sS(T)$ the collection of its separators, and by $\nu_T(S)$ the multiplicity of $S\in \sS(T)$ in $\cH^*_T$. A set $P\subseteq [m]$ is called a \textit{parent set} in $\cH^*$ if it is non-empty and there exists a component $T\in \sT_{{\rm core}}(\cH^*)$ with $P=p_{\cH^*}(T)$. The multiplicity $\tau(P)$ of $P$ is the number of $T\in \sT_{{\rm core}}(\cH^*)$ with $P=p_{\cH^*}(T)$. The collection of all parent sets in $\cH^*$ is denoted by $\sP_{{\rm core}}(\cH^*)$. Finally, by $i(\cH^*)$ we denote the number of initial components of $\cH^*$, that is the components $T\in \sT_{{\rm core}}(\cH^*)$ such that $p_{\cH^*}(T)=\emptyset$.

We refer for the following result to  \cite[Lemma 5.1]{studenyessential}.
\begin{lem}\label{lem:newuG}Let $\cH^*$ be the essential graph of a DAG $\cH$. If ${\rm core}(\cH^*)=\emptyset$ then $u_{\cH}=0$. If ${\rm core}({\cH^*})\neq\emptyset$ then the standard imset for ${\cH}$ has the form
\begin{eqnarray*}
u_{\cH}&=&\delta_{{\rm core}({\cH^*})}-\sum_{T\in \sT_{{\rm core}}({\cH^*})}\sum_{C\in \sC(T)}\delta_{C\cup p_{{\cH^*}}(T)}+\sum_{T\in \sT_{{\rm core}}({\cH^*})}\sum_{S\in \sS(T)}\nu_T(S)\delta_{S\cup p_{{\cH^*}}(T)}+\\
&+&\sum_{P\in \mathcal{P}_{{\rm core}}({\cH^*})}\tau(P)\delta_{P}+(i({\cH^*})-1)\delta_\emptyset.
\end{eqnarray*}
\end{lem}
By Lemma 5.2 in \cite{studenyessential}, unless $\cH^{*}$ is a complete graph, the terms in the above formula never cancel each other. In particular the support of $u_\cH$ is the collection of all sets of the form:
\begin{itemize}
\item[(i)] the core of $\cH^{*}$
\item[(ii)] $C\cup p^*(T)$ for $T\in \sT_{{\rm core}}(\cH^*)$ and $C\in \sC(T)$
\item[(iii)] $S\cup p^*(T)$ for $T\in \sT_{{\rm core}}(\cH^*)$ and $S\in \sS(T)$
\item[(iv)] $P$ for $P\in \mathcal{P}_{{\rm core}}({\cH^*})$
\end{itemize}
The empty set may or may now appear in the support set of $u_{\cH}$ but this does not play any role in the following arguments.
\begin{prop}\label{th:uforDAGS}Let $\cH$ be a DAG. Then $N^{*}(i)\subseteq N^{*}(j)$ if and only if $i\in A$ implies $j\in A$ for every $A$ in the support of $u_\cH$.
\end{prop}
\begin{proof}Lemma \ref{lem:newuG} gives the support of $u_\cH$ in terms of $\cH^*$, see also items (i)-(iv) above. For the forward direction first note that if $i\in C$ then $j\in C\cup p^{*}(T)$, which follows immediately from $i\in N^{*}(j)$. This implies that if $i$ lies in the core  then $j$ also lies in the core. Suppose now that $i\in C\cup p^{*}(T)$ for some $T\in \sT_{{\rm core}}(\cH^*)$ and $C\in \sC(T)$. If $i\in C$ then we have just shown that  $j\in C\cup p^{*}(T)$. If $i\in p^{*}(T)$ then $j\in p^{*}(T)$ because $c^{*}(i)\subseteq c^{*}(j)$. The arguments for the subsets of type (iii) and (iv) above  are the same.

For the opposite direction first note that if $i\in A$ implies $j\in A$ for all $A$ in the support of $u_\cH$ then taking $A=C\cup p^*(T)$ where $T$ is the connected component of $i$  and $C\in \sC(T)$ we find that either $i-j$ or $j\to i$ and hence $i\in N^{*}(j)$. Let $k\in n^*(i)\cup c^*(i)$. Suppose first that $i-j$. If $k\in  n^*(i)$ then $k\in n^*(j)$. To see that take any $C\cup p^*(T)$ such that $i,k\in C$, which implies that $j\in C$. Similarly, if $k\in c^*(i)$ then $k\in c^*(j)$, which follows by considering $P$ a parent set of the component containing $k$. Consequently $N^{*}(i)\subseteq N^{*}(j)$. The case $j\to i$ is similar.  
\end{proof}

Proposition \ref{th:uforDAGS} gives an efficient procedure of checking when $N^{*}(i)\subseteq N^{*}(j)$ without constructing the essential graph $\cH^*$, which gives the description of $G^0$. We present this procedure in the pseudocode below. 

\begin{algorithm2e}[H]
\SetKw{KwInn}{in}
 \SetAlgoLined
 \KwData{a DAG $\cH=([m],E)$}
 \KwResult{the set of pairs $(i,j)$ such that $N^{*}(i)\subseteq N^{*}(j)$}
 initialization\;
  \For{$i\to j$ \KwInn $\cH$}{
  add $i$ to $p_\cH(j)$\;
  }
 $u_\cH(\emptyset):=-1$, $u_\cH([m]):=1$, $\bS=\emptyset$\;
 \For{$i= 1$ \KwTo $m$}{
  $++u_\cH(p_\cH(i))$,$--u_\cH(p_\cH(i)\cup i)$\;
  add $\{p_\cH(i)\}$ and $\{p_\cH(i)\cup i\}$ to $\bS$\;
 } 
\ForAll{elements $S$ of $\bS$}{
if $u_\cH(S)=0$ then remove $S$ from $\bS$\;
}
 \For{$i= 1$ \KwTo $m$}{
  $\cE_i:=\{S\in \bS:\, i\in S\}$\;
}
 \For{$i\cdots j\in E$}{
  $N^{*}(i)\subseteq N^{*}(j)$ if and only if $\cE_i\subseteq \cE_j$\;
}
  \caption{The computation of $G^{0}$ for a DAG $\cH$}
\end{algorithm2e}
In addition note that the size of the support set of $u_{\cH^{*}}$ is $\leq 2m$. The fact that  it is $\leq 2m+2$ is obvious from (\ref{eq:uG}). But also  any initial vertex $i$ in $\cH$ will have $p_\cH(i)=\emptyset$ and hence $-\delta_\emptyset $ and $\delta_{p_\cH(i)}$ will cancel each other. It follows that the number of operation to build construct $G^{0}$ is quadratic in $m$. In fact all loops are linear in $m+|E|$ apart from the penultimate one. 

The imset $u_\cH$ gives in fact the complete description of the group $G$. 
\begin{lem}\label{lem:imsetAut}Let $\sigma$ be a permutation. Then $\sigma\in G$ if and only if $u_\cH=\sigma(u_\cH)$, where 
$$
\sigma(u_\cH)(S)=u_\cH(\sigma^{-1}(S)).
$$
Consequently, by Theorem \ref{th:main} we obtain the complete structure of $G$. 
\end{lem}
\begin{proof}This follows from the fact that $u_\cH$ is in a one-to-one correspondence with a DAG model of $\cH$.
\end{proof}
Lemma \ref{lem:imsetAut} does not provide an efficient algorithm to find the automorphism group of $\cH^{*}$, which in general is a hard problem.

\section{Special graphs and small examples}\label{sec:examples}

Some DAG models are equivalent to undirected graphical models, in which case we refer to \cite[Section 7]{DrKuZw2013} for examples.  To obtain a new set of examples we first consider two simple DAGs: the \textit{sprinkle graph} in Figure \ref{fig:sprinkle} and the \textit{Verma graph} in Figure \ref{fig:verma}.  

  \begin{figure}[htp!]
\begin{tikzpicture}
\tikzstyle{vertex}=[circle,fill=black,minimum size=5pt,inner sep=0pt]
    \node[vertex] (1) at (0,0)  [label=left:$1$] {};
    \node[vertex] (4) at (2,0) [label=above:$4$]{};
    \node[vertex] (2) at (1,.8) [label=above:$2$]{};
    \node[vertex] (3) at (1,-.8) [label=below:$3$]{};
    \node[vertex] (5) at (3,0) [label=above:$5$]{};
    \draw[->,-latex,line width=.3mm] (1) to (2);
    \draw[->,-latex,line width=.3mm] (1) to (3);
    \draw[->,-latex,line width=.3mm] (2) to (4);
    \draw[->,-latex,line width=.3mm] (3) to (4);
    \draw[->,-latex,line width=.3mm] (4) to (5);
  \end{tikzpicture}\qquad \qquad \begin{tikzpicture}
\tikzstyle{vertex}=[circle,fill=black,minimum size=5pt,inner sep=0pt]
    \node[vertex] (1) at (0,0)  [label=left:$1$] {};
    \node[vertex] (4) at (2,0) [label=above:$4$]{};
    \node[vertex] (2) at (1,.8) [label=above:$2$]{};
    \node[vertex] (3) at (1,-.8) [label=below:$3$]{};
    \node[vertex] (5) at (3,0) [label=above:$5$]{};
    \draw[line width=.3mm] (1) to (2);
    \draw[line width=.3mm] (1) to (3);
    \draw[->,-latex,line width=.3mm] (2) to (4);
    \draw[->,-latex,line width=.3mm] (3) to (4);
    \draw[->,-latex,line width=.3mm] (4) to (5);
  \end{tikzpicture}
  \caption{The sprinkle graph on the left and its essential graph on the right.}\label{fig:sprinkle}
  \end{figure}
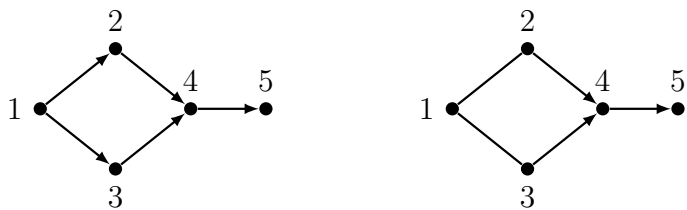

  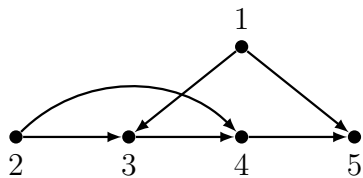
\begin{figure}[htp!]
\begin{tikzpicture}
\tikzstyle{vertex}=[circle,fill=black,minimum size=5pt,inner sep=0pt]
    \node[vertex] (1) at (0,1.2)  [label=above:$1$] {};
    \node[vertex] (2) at (-3,0) [label=below:$2$]{};
    \node[vertex] (3) at (-1.5,0) [label=below:$3$]{};
    \node[vertex] (4) at (0,0) [label=below:$4$]{};
    \node[vertex] (5) at (1.5,0) [label=below:$5$]{};
    \draw[->,-latex,line width=.3mm] (1) to (3);
    \draw[->,-latex,line width=.3mm] (1) to (5);
    \draw[->,-latex,line width=.3mm] (2) to (3);
    \draw[->,-latex,line width=.3mm,out=45,in=135] (2) to (4);
    \draw[->,-latex,line width=.3mm] (3) to (4);
    \draw[->,-latex,line width=.3mm] (4) to (5);
  \end{tikzpicture} 
  \caption{The Verma graph.}\label{fig:verma}
  \end{figure}

The essential graph of the sprinkle graph is also given in Figure \ref{fig:sprinkle}. There are no non-trivial equivalence classes and therefore $\widetilde{\cH}^*=\cH^*$. The only nontrivial relation between neighboring sets is $N^{*}(5)\subset N^{*}(4)$, so the matrices in $G^0$ have only one non-zero off-diagonal element on position $(5,4)$. The group of automorphisms of $\cH^*$ has only one non-trivial element which permutes $2$ and $3$.  Hence matrices in $G$ are in either of the two following forms:
$$
\left[\begin{array}{ccccc}
* & 0 & 0 & 0 & 0\\
0 & * & 0 & 0 & 0\\
0 & 0 & * & 0 & 0\\
0 & 0 & 0 & * & 0\\
0 & 0 & 0 & * & *\\
\end{array}\right]\qquad \mbox{and} \qquad\left[\begin{array}{ccccc}
* & 0 & 0 & 0 & 0\\
0 & 0 & * & 0 & 0\\
0 & * & 0 & 0 & 0\\
0 & 0 & 0 & * & 0\\
0 & 0 & 0 & * & *\\
\end{array}\right].
$$

The essential graph of the Verma graph $\cH$ is equal to the Verma graph itself. All equivalence classes are singletons. Moreover, there is no two distinct vertices satisfy $N^{*}(i)\subseteq N^{*}(j)$ and hence $G^0$ is equal to the group of all invertible diagonal matrices. Since there are no non-trivial automorphisms of $\cH$ then in fact the whole group $G$ consists solely of diagonal matrices. 

\tikzstyle{vertex}=[circle,fill=black,minimum size=5pt,inner sep=0pt]
\tikzstyle{vertex0}=[shape=circle,minimum size=5pt,inner sep=0pt,draw]
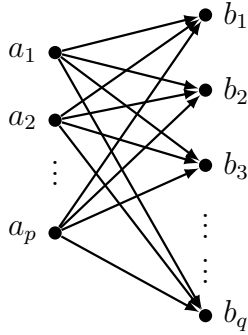
\begin{figure}[htp!]
  \begin{center}
  \begin{tikzpicture}
    \node[vertex] (1a) at (-1,1.5) [label=left:$a_{1}$]{}; 
    \node[vertex] (2a) at (-1,0.6) [label=left:$a_{2}$]{};
    \node[vertex] (4a) at (-1,-.9) [label=left:$a_{p}$]{}; 
    \node[vertex] (1b) at (1,2) [label=right:$b_{1}$]{}; 
    \node[vertex] (2b) at (1,1) [label=right:$b_{2}$]{}; 
    \node[vertex] (3b) at (1,0) [label=right:$b_{3}$]{};     
    \node[vertex] (4b) at (1,-2) [label=right:$b_{q}$]{};     
    \draw (-1,0) node{$\vdots$}; 
    \draw (1,-1.3) node{$\vdots$};  
    \draw (1,-.7) node{$\vdots$}; 
    \draw[->,-latex,line width=.3mm] (1a) to (1b);    
    \draw[->,-latex,line width=.3mm] (1a) to (2b);    
    \draw[->,-latex,line width=.3mm] (1a) to (3b);    
    \draw[->,-latex,line width=.3mm] (1a) to (4b); 
    \draw[->,-latex,line width=.3mm] (2a) to (1b);    
    \draw[->,-latex,line width=.3mm] (2a) to (2b);    
    \draw[->,-latex,line width=.3mm] (2a) to (3b);    
    \draw[->,-latex,line width=.3mm] (2a) to (4b); 
    \draw[->,-latex,line width=.3mm] (4a) to (1b);    
    \draw[->,-latex,line width=.3mm] (4a) to (2b);    
    \draw[->,-latex,line width=.3mm] (4a) to (3b);    
    \draw[->,-latex,line width=.3mm] (4a) to (4b);     
  \end{tikzpicture}
  \end{center}
  \caption{The graph of the factor model.}\label{fig:factor}
  \end{figure}

For a slightly more general example consider the DAGs defining factor models as given in Figure \ref{fig:factor}. We have $N_{\cH}(b_i)\subset N_{\cH}(a_i)$ for every $i,j$ and there are no other containment relations. The only non-zero off-diagonal elements of matrices in $G^0$ are in position $(a_i,b_j)$ for all $i,j$. For example if $p=2$ and $q=3$ then they are of the form
$$
\left[\begin{array}{cc|ccc}
* & 0 & * & * & *\\
0 & * & * & * & *\\
\hline
0 & 0 & * & 0 & 0\\
0 & 0 & 0 & * & 0\\
0 & 0 & 0 & 0 & *\\
\end{array}\right].
$$
Any automorphisms of $\cH$ is a product of any permutation permuting $\{a_1,\ldots,a_p\}$ and any permutation permuting $\{b_1,\ldots,b_q\}$. Consequently all matrices in $G$ look like the matrices in $G^0$ where the two diagonal blocks are replaced by arbitrary monomial matrices.

\appendix

\section{Proof of Theorem \ref{thm:G0}}\label{app:proof}

To prove this theorem, we will use the following two lemmas, in which $K$
is the concentration matrix of the model.

\begin{lem} \label{lm:CombCrit}
Let $A,B$ be subsets of $[m]$ of the same cardinality satisfying $j \in A$
and $i \not \in A$ and either $j \not \in B$ or else both $i,j \in B$.
If $\det K[A,B]$ is identically zero on the model but $\det K[A-j+i,B]$ is not,
then $E_{ij}\notin\liea$. 
\end{lem}

\begin{proof}
Recall that the one-parameter group $I+t E_{ij}$ acts on $K$ via
\[ K \mapsto (I-tE_{ji}) K (I-tE_{ij}). \]
In words, this matrix is obtained from $K$ by adding a multiple of the
$i$-th row to the $j$-th row and adding a multiple of the $i$-th column to the $j$-th
column. Now consider the effect of this operation on $K[A,B]$. Since
either $j \not \in B$ or else both $i,j \in B$, adding the $i$-th
column to the $j$-th has either no effect on $K[A,B]$ or else is just
an elementary column operation on $K[A,B]$. This means that it does not
affect the rank of $K[A,B]$. On the other hand, since $\det K[A-j+i,B]$
is non-zero, the rows of $K[A-j,B]$ are linearly independent, and since
$\det K[A,B]$ is zero, the $j$-th row $K[j,B]$ lies in the span of
the rows of $K[A-j,B]$. This is not true for the $i$-th row $K[i,B]$,
hence the $A \times B$-submatrix of $K+tE_{ji}K + tKE_{ij}$ has full rank
for generic $K$. This means that $I+tE_{ij}$ does not preserve the model,
hence it does not lie in the group $G$.
\end{proof}

\begin{lem}\label{lm:CombCrit2}
Let $A,B$ be subsets of $[m]$ of the same cardinality satisfying $j \in A\cap B$ and 
and $i \not \in A\cup B$. If $\det K[A,B]$ is identically zero but $\det K[A-j+i,B]+\det K[A,B-j+i]$ is not, then $E_{ij}\notin\liea$. 
\end{lem}
\begin{proof}Since $K[A,B]=0$, $E_{ij}\in\liea$ only if the determinant of the $(A,B)$-submatrix of $K_{t}:=(I-tE_{ji})K(I-t E_{ij})$ is zero. To show that it is not zero it suffices to show that the the linear term of $t$ does not vanish. To study this linear term, we alternatively study the linear term of $(I-sE_{ji})K(I-t E_{ij})$ further specializing to $s=t$. Because $E_{ij}$ has rank $1$, the determinant of the $(A,B)$-submatrix of $(I-sE_{ji})K(I-tE_{ij})$ is a  polynomial of order two in $s,t$. To find its coefficient of the linear term $s$ we can set $t=0$. Matrix $(I-sE_{ji})K$ is obtained by adding a multiple of the $i$-th row to the $j$-th row. Suppose that the elements of $A$ are $a_{1}<a_{2}<\cdots<a_{d}$ and the elements of $B$ are $b_{1}<b_{2}<\cdots <b_{d}$. Let $1\leq k\leq d$ be such that $j=a_{k}$. The determinant if its $(A,B)$-submatrix can be computed by expanding along the $k$-th row (which corresponds to the $j$-th row of $K$):
\begin{equation*}
\begin{split}
\det ((I-sE_{ji})K)[A,B]&=\sum_{l=1}^{d} (-1)^{k+l} (K_{jb_{l}}-sK_{ib_{l}})\det K[A-j,B-b_{l}]=\\
&= \det K[A,B]-s\det K[A-j+i,B].
\end{split}
\end{equation*}
Similar computations for the coefficient of $t$ give
$$
\det (K(I-tE_{ij}))[A,B]= \det K[A,B]-t\det K[A,B-j+i].
$$
Hence the coefficient of $t$ in the determinant of $K_{t}[A,B]$ is $-\det K[A-j+i,B]-\det K[A,B-j+i]$. If this sum does not identically vanish on the model then $E_{ij}\notin\liea$.
\end{proof}


Lemma~\ref{lm:G0lower} gives one direction of the proof of Theorem~\ref{thm:G0}; we need only prove that if $i\neq j$ and 
$N_{\cH}(i)\not\subseteq N_{\cH}(j)$, then $E_{ij}\notin\liea$.  First of all,
if there is no cup from $j$ to $i$, then $K[j,i]$ is identically zero,
while $K[i,i]$ is not. Hence $E_{ij} \not \in \liea$ (this is the special
case of Lemma~\ref{lm:CombCrit} with $A=\{j\}$ and $B=\{i\}$). Thus in
what follows we may assume that there do exist cups from $j$ to $i$.
We treat the various types of cups from $j$ to $i$ separately; in
each case, we assume that cups of the previous types do not exist.
Before we get going, we remark that, since there are no flags, for any
cup $(f,h,k,l)$ with $f \to h$ also $(f,k,k,l)$ is a cup. The following lemma will be also useful.
\begin{lem}\label{lem:Det00}
Let $u$ be a vertex in a NF-CG $\cH$. Let $D$ be the set of children of $u$ together with all their descendants. Then for every vertex $v\notin D\cup \{u\}$ such that there is no link between $u$ and $v$ we have $\det K[D\cup \{u\},D\cup \{v\}]=0$.
\end{lem}
\begin{proof}By Corollary \ref{cor:AB} it is enough to show that there is no self-avoiding cup system from $D\cup \{u\}$ to $D\cup \{v\}$. It is clear that the second element of every cup starting in $d\in D$ needs to lie in $D$ just because it is either equal to $d$ or it is equal to $d'$ such that $d\to d'$ in $\cH$. Also every cup from $u$ needs to have its second entry in $D$. Indeed, let $(u,l_{2},l_{3},l_{4})$ be such a cup. The node $l_{2}$ is either equal to $u$ or it is a child of $u$, in which case it lies in $D$. So suppose that $l_{2}=u$ and show that this leads to a contradiction. If $l_{2}=u$ then $l_{3}$ is either $u$ or a neighbor of $u$. If $l_{3}=u$ then $l_{4}$ must be a parent of $u$, which cannot be a vertex of $D$ (because otherwise there is a semi-directed cycle in $\cH$) and it cannot be $v$ because there is no arrow $v\to u$ (by assumption). If $l_{3}\in n_{\cH}(u)$  then $l_{4}$ must be a parent of $l_{3}$ and by the no flag assumption also a parent of $u$. This situation is also impossible because $l_{4}$ cannot lie in $D\cup \{v\}$.  Hence, by the pigeon hole principle, in any cup system from $D\cup\{u\}$ to $D\cup \{v\}$, two of the elements after one step coincide, and this proves the claim.
\end{proof}

In what follows we assume that $\cH$ is essential.

\subsection*{I. Vertex $i$ lies in $n_{\cH}(j)\cup c_{\cH}(j)$.} In that case there must exist 
$$l\in (n_{\cH}(i)\cup c_{\cH}(i))\setminus (n_{\cH}(j)\cup c_{\cH}(j)).$$ 
Let $D$ denote the set of all children of $l$ together with their descendants. We have $i,j\notin D$  and thus  $A:=D+j$ and $B:=D+l$ have the same cardinalities.  By Lemma \ref{lem:Det00} with $u=l$, $v=j$ we have $\det K[A,B]=0$. On the other hand, there does exist a self-avoiding cup system from $A-j+i$ to $B$ that links $i$ directly to $l$ without crossing $D$ and each $d \in D-j$ to itself
via $(d,d,d,d)$ and hence $\det K[A-j+i,B]\neq 0$ by Corollary \ref{cor:AB}. Now $E_{ij}\notin\liea$ by Lemma~\ref{lm:CombCrit}.

\subsection*{II. There is no arrow $i \to j$.} In that case let $D$ be the set of all children of $i$ together with their descendants. Set $A:=D + j$ and $B:= D+i$. By Lemma \ref{lem:Det00} $\det K[A,B]=0$.
But clearly $\det K[A-j+i,B]=\det K[B,B] \neq 0$. 

\subsection*{Mid-proof break.} We pause a moment to point out that we have used that $\cH$ has no flags,
but not yet that it is essential. This will be exploited in the following
arguments. Indeed, in the remaining cases, there must be an arrow $i
\to j$. This arrow must be essential, hence either the parents of $j$
in the undirected component $T$ of $i$ do not form a clique, or else
one of $\{i,j\}$ has a parent outside $T$ that is not a parent of the
other. We deal with these cases as follows.

\subsection*{III. There is an arrow $k \to j$ with $k$ in the component of $i$  at distance at
least $2$.} In that case let $D$ be the set of all children of $i$ together with their descendants. Set $A:=D+k$ and $B:=D+i$. By Lemma \ref{lem:Det00}
$\det K[A,B]=0$. But, as in the first case, $\det K[A-j+i,B] \neq 0$ because there is a self-avoiding cup system from $A-j+i$ to $B$ given by $(d,d,d,d)$ for $d\in D-j+i$ and $(j,j,j,k)$. Again, we conclude that $E_{ij} \not \in \liea$.

\medskip
\subsection*{IV. There is an induced subgraph like in Figure \ref{fig:headache}.} Let $D$ be the set of all children of $k$ together with their descendants. Set $A=D+k$ and $B=D+l$ and note that both $A$ and $B$ contain $j$. We again have $\det K[A,B]=0$ by Lemma \ref{lem:Det00}.  However, both $\det K[A-j+i,B]$ and $\det K[A,B-j+i]$ are nonzero. Even more: the sum of these two determinants is also nonzero because $\det K[A-j+i,B]$ has a monomial that does not appear in $\det K[A,B-j+i]$: consider the cup system from $A-j+i$ to $B$ given by $(i,i,l,l)$, $(k,j,j,j)$ and $(d,d,d,d)$ for all $d\in D-j$. By Proposition \ref{prop:monomialsInDet} this system corresponds to a monomial in $\det K[A-j+i,B]$. On the other hand this monomial cannot appear in $\det K[A,B-j+i]$ because it contains only one element of $\Lambda$, namely $\lambda_{kj}$, and only one off-diagonal element of $\Omega$, namely $\omega_{il}$. This means that it must correspond to a cup system between $A$ and $B-j+i$ that contains only one undirected edge $i-l$ and onearrow $k\to j$. However any cup from $i$ to $A$ must contain either an arrow $i\to p$ for some $p\in D$ or an undirected edge $i-k$. By Lemma~\ref{lm:CombCrit2} we conclude that $E_{ij}\notin\liea$.

\subsection*{V. There is an arrow $k \to j$ with $k \not \in T$ and no arrow
between $k$ and $i$.}
So we have the induced subgraph $i \to j \ot k$. Let $D$ be the set of all children of $i$ together with their descendants. Set $A:=D+k$ and $B:=D+i$.
By Lemma \ref{lem:Det00} $\det K[A,B] =0$. On the other hand,
$\det K[A-j+i,B] \neq 0$, because of the self-avoiding cup system from
$A-j+i$ to $B$ consisting of $(k,j,j,j)$ and $(i,i,i,i)$ and $(d,d,d,d)$
for all $d \in D-j$. Again, we may apply Lemma~\ref{lm:CombCrit}, this
time with $i,j$ both in $B$, to conclude that $E_{ij} \not \in \liea$.

\subsection*{VI. There is an arrow $l \to i$ with no arrow from $l$ to $j$.}
Pictorially, we have $l \to i \to j$. Let $D$ be the set of children of $j$ together with all their descendants. Set $A=D+j$ and $B=D+l$. By Lemma \ref{lem:Det00} we have $K[A,B]=0$. However, $K[A-j+i,B]\neq 0$ and hence $E_{ij}\notin\liea$ by Lemma~\ref{lm:CombCrit}.

\subsection*{VII. There is an arrow $i \to l$ and $l\to j$} Without loss of generality we can assume that $l$ is minimal in the sense that if $i\to l'$, $l'\to j$ then there is no arrow from $l$ to $l'$. Since $\cH$ is essential then $l\to j$ is an essential arrow. This implies one of the following possibilities:
\begin{itemize}
\item[(i)] There exists $k$ in the component of $l$ with distance at least two to $l$ and with $k\to j$ 
\item[(ii)]  There is an induced subgraph like in Figure \ref{fig:headache2}. 
\item[(iii)] There are arrows $l\to k$, $k\to j$ 
\item[(iv)] There is an arrow $k\to l$ and no arrow from $k$ to $j$.
\end{itemize}

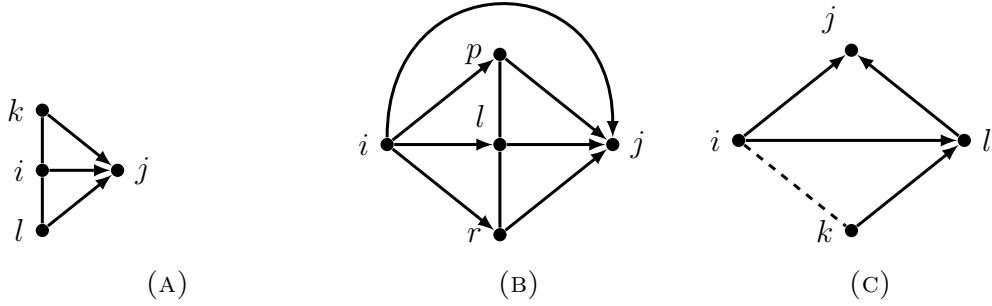
\begin{figure}
        \centering
        \begin{subfigure}[b]{0.3\textwidth}
        \begin{tikzpicture}
\tikzstyle{vertex}=[circle,fill=black,minimum size=5pt,inner sep=0pt]
    \node[vertex] (i) at (0,0)  [label=left:$i$] {};
    \node[vertex] (j) at (0,.8) [label=left:$k$]{};
    \node[vertex] (k) at (1,0) [label=right:$j$]{};
    \node[vertex] (l) at (0,-.8) [label=left:$l$]{};
    \draw[line width=.4mm] (i) to (j);
    \draw[->,-latex,line width=.4mm] (i) to (k);
    \draw[line width=.4mm] (i) to (l);
    \draw[->,-latex,line width=.4mm] (j) to (k);
    \draw[->,-latex,line width=.4mm] (l) to (k);
  \end{tikzpicture}
  \caption{}\label{fig:headache}
        \end{subfigure}%
        ~ 
        \begin{subfigure}[b]{0.3\textwidth}
                \begin{tikzpicture}
\tikzstyle{vertex}=[circle,fill=black,minimum size=5pt,inner sep=0pt]
    \node[vertex] (w) at (-1.5,0)  [label=left:$i$] {};
    \node[vertex] (i) at (0,0)  [label=120:$l$] {};
    \node[vertex] (j) at (0,1.2) [label=left:$p$]{};
    \node[vertex] (k) at (1.5,0) [label=right:$j$]{};
    \node[vertex] (l) at (0,-1.2) [label=left:$r$]{};
    \draw[line width=.4mm] (i) to (j);
    \draw[->,-latex,line width=.4mm] (i) to (k);
    \draw[line width=.4mm] (i) to (l);
    \draw[->,-latex,line width=.4mm] (j) to (k);
    \draw[->,-latex,line width=.4mm] (l) to (k);
        \draw[->,-latex,line width=.4mm,out=90,in=90,looseness=2] (w) to (k);
                \draw[->,-latex,line width=.4mm] (w) to (i);
 \draw[->,-latex,line width=.4mm] (w) to (j);
 \draw[->,-latex,line width=.4mm] (w) to (l);
  \end{tikzpicture}
                \caption{}
                \label{fig:headache2}
        \end{subfigure}
        ~ 
        \begin{subfigure}[b]{0.3\textwidth}
              \begin{tikzpicture}
\tikzstyle{vertex}=[circle,fill=black,minimum size=5pt,inner sep=0pt]
    \node[vertex] (i) at (-1.5,0)  [label=left:$i$] {};
    \node[vertex] (j) at (0,1.2)  [label=120:$j$] {};
    \node[vertex] (l) at (1.5,0) [label=right:$l$]{};
    \node[vertex] (k) at (0,-1.2) [label=left:$k$]{};
    \draw[line width=.4mm,dashed] (i) to (k);
    \draw[->,-latex,line width=.4mm] (i) to (j);
    \draw[->,-latex,line width=.4mm] (i) to (l);
        \draw[->,-latex,line width=.4mm] (l) to (j);
    \draw[->,-latex,line width=.4mm] (k) to (l);
  \end{tikzpicture}
  \caption{}\label{fig:headache3}
        \end{subfigure}
        \caption{Some special induced subgraphs considered in the proof.}\label{fig:animals}
\end{figure}

\subsection*{VII.(i)} In this case we have an induced subgraph $k\to j\ot l$. Let $D$ be the set of children of $l$ and all their descendants. Set $A=D+l$ and $B=D+k$. The argument that $\det K[A,B]=0$ is the same as in the previous cases. By Lemma \ref{lm:CombCrit2} $E_{ij}\notin\liea$ because $\det K[A-j+i,B]+\det K[A,B-j+i]\neq 0$. To verify this last statement note that by Proposition \ref{prop:monomialsInDet} $\det K[A-j+i,B]$ contains a monomial corresponding to the cup system $(d,d,d,d)$ for $d\in D-j$, $(i,k,k,k)$ and $(l,j,j,j)$. This monomial contains $\lambda_{ki}$, $\lambda_{lj}$ and no off-diagonal $\omega$'s. There is no cup system from $A$ to $B-j+i$ that uses only $k\to j$ and $l\to j$ and hence this monomial does not appear in $\det K[A,B-j+i]$. 

\subsection*{VII.(ii)} Let $D$ be the set of children of $p$ and all their descendants. Set $A=D+p$ and $B=D+r$. Again $\det K[A,B]=0$ but  $\det K[A-j+i,B]+\det K[A,B-j+i]\neq 0$. For this, we note that $\det K[A-j+i,B]$ contains a monomial corresponding to the cup system $(d,d,d,d)$ for $d\in D-j$, $(i,r,r,r)$ and $(p,j,j,j)$, which does not appear in $\det K[A,B-j+i]$. Now $E_{ij}\notin\liea$ by Lemma \ref{lm:CombCrit2}.

\subsection*{VII.(iii)} Note that in this case no link between $i$ and $k$ is possible (by maximality of $l$ and no semi-directed cycle assumption).  But then $E_{ij}\notin\liea$ by Case V.

\subsection*{VII.(iv)} Note that in this case by case VI. the arrow $k\to i$ is impossible and thus we have either $i-k$, $i\to k$ are there is no link between them. The induced subgraph is given in Figure \ref{fig:headache3}, where the dashed edge indicate the three possibilities for the link between $i$ and $k$. Let $D$ be the set of children of $j$ together with all descendants. Set $A=D+j$, $B=D+k$. Again by Lemma \ref{lem:Det00} we have that $\det K[A,B]=0$. Moreover, $\det K[A-j+i,B]\neq 0$. Now $E_{ij}\notin\liea$ by Lemma \ref{lm:CombCrit}.
This exhausts all possible cases and hence finishes the proof.

\bibliographystyle{amsalpha}
\bibliography{../!bibliografie/algebraic_statistics}
\end{document}